\documentclass[oneside]{amsart}
\usepackage[stable]{footmisc}

\usepackage{setspace}
\usepackage[a4paper,marginparwidth=2cm,,outer=4cm]{geometry}
\usepackage{amsmath,amsthm,amssymb,array,enumerate,parskip,graphicx,etoolbox,tabularx,mathtools}
\usepackage[all,arc]{xy}
\usepackage{enumitem}

\usepackage{xcolor} 

\usepackage[pagebackref]{hyperref}
\definecolor{dark-red}{rgb}{0.5,0.15,0.15}
\definecolor{dark-blue}{rgb}{0.15,0.15,0.6}
\definecolor{dark-green}{rgb}{0.15,0.6,0.15}
\hypersetup{
    colorlinks, linkcolor=dark-red,
    citecolor=dark-blue, urlcolor=dark-green
}

\renewcommand*{\backref}[1]{}
\renewcommand*{\backrefalt}[4]{%
 \ifcase #1 %
No citations.
  \or
(cit. on p. #2).%
  \else
(cit on pp. #2).%
  \fi%
}
\numberwithin{equation}{section}

\usepackage[nameinlink,capitalise,noabbrev]{cleveref}
\newcount\tmp
\newcommand{\Z}{\mathbb{Z}}
\newcommand{\Q}{\mathbb{Q}}
\newcommand{\G}{\mathbb{G}}
\newcommand{\R}{\mathbb{R}}
\newcommand{\W}{\mathbb{W}}
\newcommand{\F}{\mathbb{F}}
\newcommand{\xr}{\xrightarrow}
\renewcommand{\S}{\mathbb{S}}
\newcommand{\Hom}{\operatorname{Hom}}
\newcommand{\Res}{\operatorname{Res}}
\newcommand{\Coind}{\operatorname{Coind}}
\newcommand{\End}{\operatorname{End}}

\DeclareMathOperator{\Aut}{Aut}
\DeclareMathOperator{\Gal}{Gal}

\newtheorem{theorem}{Theorem}[section]
\newtheorem{conjecture}[theorem]{Conjecture}
\newtheorem{lemma}[theorem]{Lemma} 
\newtheorem{cor}[theorem]{Corollary}
\newtheorem{example}[theorem]{Example}
\newtheorem{prop}[theorem]{Proposition} \theoremstyle{definition}
\newtheorem{rem}[theorem]{Remark} 
\newtheorem{definition}[theorem]{Definition}
 
\usepackage{cleveref}
\Crefname{cor}{Corollary}{Corollaries}
\Crefname{conjecture}{Conjecture}{Conjectures}
\Crefname{rem}{Remark}{Remarks}
\Crefname{question}{Question}{Questions}
\Crefname{prop}{Proposition}{Propositions}

\AtBeginDocument{%
   \def\MR#1{}
}

\begin{document}
\title{The Tate spectrum of the higher real $K$-theories at height $n=p-1$}
\author{Drew Heard}
\address{Max Planck Institute for Mathematics, Bonn, Germany}
\email{drew.heard@mpim-bonn.mpg.de}
\date{\today}
\begin{abstract}
Let $E_n$ be Morava $E$-theory and let $G \subset \G_n$ be a finite subgroup of $\G_n$, the extended Morava stabilizer group. Let $E_{n}^{tG}$ be the Tate spectrum, defined as the cofiber of the norm map $N:(E_n)_{hG} \to E_n^{hG}$. We use the Tate spectral sequence to calculate $\pi_*E_{p-1}^{tG}$ for $G$ a maximal finite $p$-subgroup, and $p$ an odd prime. We show that $E_{p-1}^{tG} \simeq \ast$, so that the norm map gives equivalence between homotopy fixed point and homotopy orbit spectra. Our methods also give a calculation of $\pi_*E_{p-1}^{hG}$, which is a folklore calculation of Hopkins and Miller.
\end{abstract}
\maketitle

\section{Introduction}
Let $E_n$ be the $n$-th Morava $E$-theory (always considered at a prime $p$, which is suppressed from the notation) with coefficient ring
\[
\pi_*(E_n) = \W(\F_{p^n})[u_1,\ldots,u_{n-1}][\![u^{\pm 1}]\!]
\]
Here $|u_i| =0,|u| = -2$ and   $\W(\F_{p^n})$ is the ring of Witt vectors over the finite field $\F_{p^n}$. By work of Goerss--Hopkins--Miller~\cite{rezkhm,gh04}, the spectrum $E_n$ has an $E_\infty$-action by the Morava stabilizer group, denoted $\G_n$. These theories are complex oriented, and support formal group laws of height $n$. 

 Work of Devinatz and Hopkins, Davis and others~\cite{devhop04,davis06} shows that, for a closed subgroups $K \subset \G_n$, it is possible to define a homotopy fixed point spectrum $E_n^{hK}$ (or at least a spectrum that behaves like a homotopy fixed point spectrum), although we will only ever consider finite subgroups, for which the usual homotopy fixed point construction works (and agrees with the more complicated constructions of Devinatz--Hopkins and Davis). 

 We will work at height $n=p-1$ as here $\G_n$ has infinite cohomological dimension and has interesting maximal finite $p$-subgroups.  Indeed, the maximal finite subgroups (of the closely related group $\S_n$, as well as $\G_n$) have been classified completely by Hewett and Bujard \cite{hewett,bujard2012finite}. At height $n=p-1$, up to conjugation, there is a unique maximal finite subgroup $G$ whose order is divisible by $p$. We note that if $n$ is is not divisible by $(p-1)$, then all maximal finite subgroups have order coprime to $p$ (see~\cite[Theorem 1.3]{hewett} or~\cite[Proposition 1.7]{bujard2012finite}). 

  Recall that for finite subgroups we can form the homotopy orbit spectrum $(E_n)_{hK}$ as $EK_+ \wedge_K E_n$, where $EK$ is, as usual a free, contractible $K$-CW complex, and there is a norm map $N:(E_n)_{hK} \to (E_n)^{hK}$ relating them (see~\Cref{sec:tate}). Define $E_n^{tK}$, the Tate spectrum, to be the cofiber of the norm map. 
Our main result concerns the Tate spectrum associated to the maximal finite $p$-subgroup $E_{p-1}^{tG}$, for $p$ an odd prime. 
\begin{theorem}\label{thm:1}
	Let $p>2$ be an odd prime. The Tate spectrum $E_{p-1}^{tG}$ is contractible and so the norm map $N:(E_{p-1})_{hG} \to E_{p-1}^{hG}$ is an equivalence.
\end{theorem}  
\begin{rem}
	This is also true at $n=1,p=2$ where $E_1^{hC_2}$ can be identified with $2$-adic real $K$-theory (see~\cite{koduality}), although our calculations do not cover this case. 
\end{rem}

It is likely that this result is well-known to the experts, although we have not found a proof in the literature. It is well known that~\Cref{thm:1} holds $K(n)$-locally (for all finite subgroups $K$ and at all heights); one proof follows from the fact that $E_n^{hK} \to E_n$ is a faithful $K(n)$-local $K$-Galois extension and~\cite[Proposition 6.3.3]{rognesgalois}. We prove that, in this particular case, it is true even before any localisation.\footnote{Note that $E_n^{hK}$ is always $K(n)$-local (since $E_n$ is), but the same is not (necessarily) true for $(E_n)_{hK}$ and $E_n^{tK}$.}

 We recall from the discussion above that when $n$ is not a multiple of $p-1$ there are no non-trivial finite $p$-subgroups. In this case a simple argument shows that $E_n^{tK}$ is also contractible (see~\Cref{prop:coprimevanish}).  In fact, using derived algebraic geometry and the affineness machinery of Mathew and Meier~\cite{mmaffine}, Meier has shown the author a proof that the vanishing of the Tate spectrum holds more generally for any finite subgroup $K \subset \G_n$.  

Our method is by brute force; we simply calculate $\pi_*E_{p-1}^{tG}$ using the Tate spectral sequence and show that it vanishes. In fact, the calculation of $\widehat{H}^*(G;(E_{p-1})_*)$ is essentially done in~\cite{MR2066512}, although Symonds does not use this to calculate $\pi_*E_{p-1}^{tG}$. One benefit of this computational approach is that along the way we give a detailed calculation of the homotopy fixed point spectral sequence for $\pi_*E_{p-1}^{hG}$. This calculation has quite a history; it appears to have first been calculated (in unpublished work) by Hopkins and Miller in the '90s; details appeared in Nave's thesis~\cite{nave}, with further details published in~\cite{naveannals}. These calculations have also appeared in the thesis of Hill~\cite{hillthesis} and work of Henn~\cite{HennFinite}. The case of $n=2,p=3$ was considered in detail in~\cite{GHMR,skypaper}, and it is their approach that we follow. Whilst we do not claim any originality for this result, we hope the reader finds it useful to see the calculation written down in detail. 

 The computational method also has the advantage of allowing an easy description of (most of) the homotopy orbit spectral sequence. This can be useful; for example with Stojanoska~\cite{koduality} we used a vanishing result for the Tate spectrum of $KU^{tC_2}$ to calculate the Anderson dual of $KO$. In particular, the spectral sequence for computing the Anderson dual of $KO$ was shown to be the linear dual of the homotopy orbit spectral sequence for $KO$. A similar result is also true for $Tmf[1/2]$~\cite{stojanoska2012duality}. Behrens and Ormsby~\cite{behrensormsby} also use the vanishing of the Tate spectrum $\text{TMF}_1(5)^{t\F_5^\times}$ and the resulting calculation of the homotopy orbit spectral sequence to compute hidden extensions in the homotopy fixed point spectral sequence for $\text{TMF}_0(5)$. 

As an application of these results we have the following partial description of the $K(n)$-local Spanier-Whitehead dual of $E_n^{hG}$, i.e.,
\[
DE_n^{hG} \coloneqq F(E_n^{hG},L_{K(n)}S^0). 
\] 
\begin{cor}
        $E_{p-1}^{hG}$ is $K(p-1)$-locally self-dual up to suspension. In fact $DE_{p-1}^{hG} \simeq \Sigma^N E_{p-1}^{hG}$ where $N \equiv -(p-1)^2 \mod (2p(p-1)^2)$, and $N$ is only uniquely defined modulo $2p^2(p-1)^2$.  
\end{cor}
We can, in fact, show that $E_n^{hK}$ is always self-dual up to some suspension, although we cannot be quite so precise as above, unless $K$ is coprime to $p$, in which case $DE_n^{hK} \simeq \Sigma^{-n^2}E_n^{hK}$. 

\section*{Acknowledgments}
This work formed part of the author's PhD at Melbourne University under Craig Westerland. It is a pleasure to again acknowledge the encouragement and support of Craig, for which this work would not exist without. Many of the techniques in this paper we learnt from~\cite{skypaper}. We also thank Hans-Werner Henn and Justin Noel for helpful comments. The preparation of this note was supported by a David Hay postgraduate writing-up award provided by Melbourne University. We thank the Max Planck Institute in Bonn for its hospitality.
 \section{Finite subgroups of the Morava stabilizer group}
Let $\Gamma_n$ denote the height $n$ Honda formal group law over $\F_{p^n}$ with $p$-series $[p](x) = x^{p^n}$. We let $\S_n \coloneqq \operatorname{Aut}(\Gamma_n)$ denote the $n$-th Morava stabilizer group. By Lubin-Tate theory~\cite{lubtate} there is a ring $(E_n)_0 = \W(\F_{p^n})[\![u_1,\ldots,u_{n-1}]\!]$  with an action of $\S_n$, where $\W(\F_{p^n})$ denotes the Witt vectors, whose associated formal group law is the universal deformation of the Honda formal group law. By the Landweber exact functor theorem there is a cohomology theory with coefficient ring $(E_n)_* = (E_n)_0[u^{\pm 1}]$, where $|u|=-2$, and the action of $\S_n$ extends to an action on $(E_n)_*$. We let $E_n$ denote the representing spectrum.  The Honda formal group law is defined over the field $\F_{p^n}$, and so we can consider the group $\G_n = \S_n \rtimes \Gal(\F_{p^n}/\F_{p})$, known as the extended Morava stabilizer group. This also acts on $(E_n)_*$, and hence $E_n$,  and furthermore the Goerss--Hopkins--Miller theorem~\cite{rezkhm,gh04} proves that this action can be lifted to a spectrum level action via $E_\infty$-ring maps.
 
 There is an alternative description of ring $\S_n$. Write $\W$ for the Witt vectors over $\F_{p^n}$ and define a ring
 \[
 \mathcal{O}_n = \W\langle S \rangle/(S^n-p),
 \]
 where the angled brackets refer to adjoining a non-commuting power series variable $S$ subject to the condition that 
 \[
Sa = a^\phi S,
 \]
 for $a \in \W$ and $a^\phi:\W \to \W$ is the lift of Frobenius to the Witt vectors. The algebra $\mathcal{O}_n$ is a free rank $n$ module over $\W$ with generators $1, S, \ldots,S^{n-1}$, and is
the ring of integers in a division algebra $\mathbb{D}_n$ over $\Q_p$ of dimension $n^2$ and Hasse invariant
$1/n$. With this notation, $\S_n = \mathcal{O}_n^\times$. 
 
We will be interested in the maximal finite $p$-subgroups (up to conjugacy) of $\G_n$ at height $n=p-1$; as mentioned previously, this is an important case due to the presence of $p$-torsion in $\G_n$, which implies that the cohomological dimension of $\G_n$ is infinite.  These sometimes go by the name of the higher real $K$-theories (although we refrain from using this), and are sometimes denoted $EO_{n}$, for reasons that the following simple example explains. 
\begin{example}
Let $n=1$ and $p=2$. For $n=1$, the Honda formal group law over $\F_2$ is isomorphic to the multiplicative formal group law:
\[
\G_m(x,y) = x+y-xy.
\]

There is an isomorphism $\End_{\F_2}(\G_m) \simeq \Z_2$,
 and so $\Aut_{\F_2}(\G_m) = \Z_2^\times \simeq \Z_2 \times C_2$, where $C_2$ is generated by -1. The universal deformation of $\G_m$ is once again $\G_m$, now considered over $\Z_2$. Finally then, by adjoining a degree 2 element $u$, we see that $E_1$ is nothing other than $p$-completed complex $K$-theory, and the action of $\G_1$ is just the ($p$-adic) Adams operations. 
 
 There is a unique maximal finite subgroup of $\Z_2^\times$, namely $C_2 \simeq \{ \pm 1\}$, and so we can form the homotopy fixed point spectrum $E_1^{hC_2}$. By calculating the homotopy groups of this spectrum  we can see that this is actually 2-complete real $K$-theory; $E_1^{hC_2} \simeq KO^{\wedge}_2$ (see, for example,~\cite[5.5.2]{rognesgalois}).  

 \end{example}
We now move to the more general case, although we fix $n=p-1$ for the remainder of this section, where $p$ is an odd prime. The maximal finite subgroups of $\S_n$ were first classified by Hewett, in his work on subgroups of finite division algebras~\cite{hewett}, whilst the case of $\S_n$ and $\G_n$ has also been studied by Bujard~\cite{bujard2012finite}. We will be interested in finite subgroups of order divisible by $p$.  

By~\cite[Theorem 1.3.1]{bujard2012finite} or~\cite[Theorem 1.3]{hewett}, when $n=p-1$ there are exactly 2 conjugacy classes of maximal finite subgroups of $\S_n$, represented by
\[
F_0 = C_{p^n-1}
\]
and
\[
 F=F_1 = C_p \rtimes C_{n^2} \simeq C_p \rtimes C_{{(p-1)}^2},
\]
where the action of the right factor on the left is given by the mod $(p-1)$ reduction map $C_{(p-1)^2} \to C_{{p-1}} \simeq \operatorname{Aut}(C_p)$.

 We choose a presentation of the group $F$ as the following
 \[
 F = \langle \zeta,\tau | \zeta^p=1,\tau^{n^2}=1,\tau^{-1} \zeta \tau = \zeta^e \rangle,
 \]
where $e \in (\Z/p)^\times$ is a generator.

It is worthwhile to see how these finite subgroups arise. Again, we refer the reader to the standard references~\cite{hewett,bujard2012finite} for more detailed information, as well as~\cite[Section 3.6]{HennFinite}. 

Let $\omega \in \W^{\times} \subset \S_n$ be a $(p^n-1)$-st root of unity. Then $X \coloneqq \omega^{{(p-1)}/2}S \in \mathbb{D}_n$ is an element such that $X^n = -p$. It can be shown that the subfield $\Q_p(X) \in \mathbb{D}_n$ is isomorphic to the cyclotomic extension $\Q_p(\zeta)$ generated by a primitive $p$-th root of unity $\zeta$~\cite[Lemma 19]{HennFinite}. A simple check shows that $X \omega X^{-1} = \omega^p$. Then, if we let $\eta = \omega^{(p^n-1)/n^2} \in \W$, we have \[
\eta X \eta^{-1} = \eta^{-n} X;
\]
since $\eta^{-n}$ is an $n$-th root of unity, conjugation by $\eta$ induces an automorphism of $\Q_p(X)$.  

 We will write $\tau$ for the image of $\eta$ in $\S_n$. The subgroup generated by $\zeta$ and $\tau$ is precisely the maximal finite $p$-subgroup of $F$ of $\S_n$. That this is unique up to conjugacy follows from an argument using the Skolem-Noether theorem, as in~\cite[Example 1.33]{bujard2012finite}. This relies on the fact that $X \eta X^{-1} = \eta^p$, which implies that the subgroup $F$ is normalised by $X$ (note that $X$ commutes with $\zeta$).
 
 \begin{example}
	Let $n=2,p=3$ and $\omega$ be a primitive 8-th root of unity. Then the element $\zeta=-\frac{1}{2}(1+\omega S)$ is an element of order 3 and $\omega^2 \zeta \omega^{-2} =\zeta^2$, so that $F$ is generated by $\zeta$ and $\tau:= \omega^2$. Here $F$ usually goes by the name $G_{12}$ and is abstractly isomorphic to the non-trivial semi-direct product $C_3 \rtimes C_4$.
\end{example}

We wish to extend this to $\G_n$; that is, to find a group $G$ such that there is an extension of the form
\[
1 \to F \to G \to \Gal(\F_{p^n}/\F_p) \to 1
\]
where $F$ is a maximal finite subgroup in $\S_n$. This is, in fact, thoroughly analysed in~\cite[Chapter 4]{bujard2012finite}, although we can analyse our case more simply, following~\cite{MR2066512}. Let $N_{\S_n}(\langle \zeta \rangle)$ be the normalizer of the subgroup of order $p$ inside $\S_n$.  Every conjugate of $F$ inside $\G_n$ is in fact in $\S_n$~\cite{MR2066512}; from the discussion above it is therefore conjugate to $F$ inside $\S_n$. This implies that $N_{\G_n}(F)/N_{\S_n}(F) \simeq \Gal$. We choose $c$ to have image $\sigma$ (the Frobenius), under this isomorphism. Then $F$ and $c$ generate a subgroup $G$ of order $pn^3$. In fact, as in~\cite{MR2066512} we can choose $c$ such that it acts trivially on $\langle \zeta \rangle$ and for $\langle c \rangle$ to have no pro-$p$ part. 

\begin{example}
	Again let $n=2,p=3$ and define $\psi = \omega \phi$, where $\phi$ is the generator of the Galois group. The group generated by $\psi,\zeta$ and $\tau$ goes by the name $G_{24}$. The group generated by $\tau$ and $\psi$ is the quaternion group $Q_8$ and there is an abstract isomorphism $G_{24} \simeq C_3 \ltimes Q_8$.
\end{example}
\section{The Tate spectrum}\label{sec:tate}
We start by reviewing some preliminaries on equivariant stable homotopy theory; for more details we direct the reader to the usual references such as~\cite{lms86}. As usual, if $V$ is a representation of a group $G$, we let $S^V$ denote the one-point compactification of $\R^n$ as a $G$-space. We then let $\Sigma^V(-) \coloneqq S^V \wedge -$ and $\Omega^v(-) \coloneqq F(S^V,-)$ be the suspension and loop space functors respectively.  
\begin{definition}
    A $G$-universe $U$ is a countably infinite real inner product space with an action of $G$ through linear isometries such that:
    \begin{enumerate}
         \item $U$ is the sum of countably many copies of a set of representations of $G$; and,
         \item $U$ contains the trivial representation.
     \end{enumerate} \end{definition}
If $U$ contains a copy of every irreducible representation of $G$, then $U$ is called complete; conversely, it is called trivial if it only contains the trivial representations. 

A $G$-spectrum indexed on $U$ consists of based $G$-spaces $EV$ for each indexed space $V \in U$ together with a system of based $G$-homeomoprhisms
\[
\tilde\sigma: EV \xr{\simeq} \Omega^{W-V}EW
\]
for $V \subset W$. 

A genuine $G$-spectrum is one indexed on a complete universe; a naive $G$-spectrum is one indexed on a trivial universe. In particular, a naive $G$-spectrum consists of a sequence of based left $G$-spaces $E_k$, along with based $G$-maps $\Sigma E_k \to E_{k+1}$, whose adjoints are homeomorphisms. Let us fix an identification $U^G = \R^\infty$ and write $G\mathcal{SU}$ and $G\mathcal{R^\infty}$ for the category of genuine and naive $G$-spectra spectra respectively. 

There is a restriction of universe functor $i^\ast:G\mathcal{SU} \to G\mathcal{R^\infty}$ from genuine $G$-spectra to naive $G$-spectra which has a left adjoint $i_*$~\cite[Chapter II]{lms86}. Likewise, there is a functor $\text{Sp} \to G\mathcal{R^\infty}$ from spectra to naive $G$-spectra, giving a spectrum trivial $G$-action; this has a right adjoint given by fixed points $(-)^G$. There is one more adjunction we need to consider; naive $G$-spectra can be considered as presheaves of spectra on $\mathcal{O}(G)$, the orbit category of $G$. The spectrum $E_n$ with its $G$ action is an object of the category of presheaves of spectra on $BG$; the forgetful functor from naive $G$-spectra to this later category has a left adjoint, and we will always implicitly use this to give $E_n$ the structure of a naive $G$-spectrum.

For a naive $G$-spectrum $E$, we define homotopy orbit spectra and homotopy fixed points of spectra in the usual way:
\[
E_{hG} = (EG_+ \wedge Y)/G, \quad \text{ and } \quad E^{hG} = F(EG_+,Y)^G.
\]
Following~\cite{greenlees1995generalized,ademcohendwyer,weisswilliams} (amongst other sources) there is a norm map\footnote{A comparison of these norm maps is provided in~\cite[Section I.V]{greenlees1995generalized}.}
\[
N(E): E_{hG} \to E^{hG}
\]
between homotopy orbits and homotopy fixed points. 

Based on work of Klein~\cite{klein}, Kuhn has proved the following characterisation of the norm map.
\begin{prop}\cite[Proposition 2.3]{kuhntatecohom}
     Let $N(E),N'(E):E_{hG} \to E^{hG}$ be natural transformations such that both $N(\Sigma^\infty G_+)$ and $N' (\Sigma^\infty G_+)$ are weak equivalences. Then there is a unique weak natural equivalence $f(E):E_{hG} \xr{\sim} E_{hG}$ such that the following diagram commutes:
\[
\xymatrix{
    E_{hG} \ar[r]^{N(E)} \ar[dr]_{f(E)} & E^{hG} \\
    & E_{hG}. \ar[u]_{N'(E)}
}
     \]
 \end{prop} 
 \begin{definition}
    Let the Tate spectrum, $E^{tG}$ be the cofiber of $N(E):E_{hG} \to E^{hG}$. 
\end{definition}
Let $E\tilde G$ be defined via the cofiber sequence
\[
EG_+ \to S^0 \to \tilde{E} G,
\]  where the first map sends $EG$ to the non-base point of $S^0$. By work of Greenlees and May~\cite{greenlees1995generalized}, $E^{tG}$ can equivalently be defined by
\[
E^{tG} \coloneqq (\tilde EG \wedge F(EG_+,X))^G,
\] 
although what we call $E^{tG}$ would be referred to be Greenlees and May as $t_G(E)^G$. Taking $G$-fixed points ensures that $E^{tG}$ ends up in the category of ordinary, non-equivariant, spectra. Note that although Greenlees and May work with genuine $G$-spectra,~\cite[Lemma I.1.3]{greenlees1995generalized} implies that the homotopy type of $t_E(G)$ depends only on the underlying naive $G$-spectrum of $E$ (see also~\cite[pp. 884]{snlrognes}).

Applying the results of~\cite{greenlees1995generalized} there is then a spectral sequence
\[
\hat H^s(G,\pi_tE) \Rightarrow \pi_{t-s}E^{tG},
\]
where $\hat H$ denotes Tate cohomology~\cite{cassfroh}. When $E$ is a ring spectrum, so are $E^{hG}$ and $E^{tG}$, and the homotopy fixed point and Tate spectral sequences are spectral sequences of differential algebras, with a natural map between them that is compatible with the differential algebra structure. Proofs of these claims can be found in~\cite{greenlees1995generalized}. 

We can already prove a vanishing result for the homotopy of the Tate spectrum $E_n^{tG}$ for a large class of finite subgroups of $\G_n$. For example, as noted in the introduction, when $p \gg n$ this applies to all finite subgroups. 
\begin{prop}\label{prop:coprimevanish}
    If $G \subset \G_n$ is a finite subgroup with $|G|$ coprime to $p$, then $E_n^{tG} \simeq \ast$. 
\begin{proof}
    Since $|G|$ is coprime to $p$, $\pi_*((E_n)_{hG}) = (\pi_*E_n)_G$ and $\pi_*(E_n^{hG}) = (\pi_*E_n)^G$. There is a map
    \[
\pi_*E_n \to (\pi_*E_n)_G \xr{N} (\pi_*E_n)^G \to \pi_*E_n,
    \]
where $N$ is the usual norm map from coinvariants to invariants. This map is given given by multiplication by $|G|$, which is invertible, and hence the norm map must be an equivalence.  
\end{proof}
\end{prop}
\section{A description of $E_n$ as a $G$-module}\label{sec:eng}
Hopkins and Devinatz~\cite{hd95} have calculated the action of the ring $\S_n$ on $(E_n)_*$, but the formulas are difficult to compute with directly. The calculations of Nave and others rely on the idea that there is a nicer presentation of $(E_n)_*$ for which the action of a finite subgroup $K \subset \G_n$ is easier to describe. 

The following conjecture of Hopkins was presented by Mike Hill at Oberwolfach.  
\begin{conjecture}[Hopkins, \cite{hillMFO}]\label{conj:hopkins}
    If $K \subset \S_n$ is a finite subgroup, then there is $K$-equivariant isomorphism
    \[
    (E_n)_*\simeq S_{\W(\F_{p^n})}(\rho)[N^{-1}]^\wedge_I,
    \]
    where $\rho \in E_n$ is placed in degree -2, $S(\rho)$ denotes the graded symmetric algebra, $N$ is a trivial representation corresponding to the multiplicative norm over the group on $\rho$ and $I$ is an ideal in degree 0. 
\end{conjecture}
\begin{rem}
    It appears that Hill, Hopkins, and Ravenel have investigated this conjecture further when $n=k(p-1)$ and have information on $\pi_*E_{k(p-1)}^{hC_p}$; see~\cite{hhreo}.
\end{rem}

In the case of $n=p-1$~\Cref{conj:hopkins} has been verified when $K = G$ is the maximal finite subgroup of $\G_n$, although there does not appear to be a proof in the literature.

\begin{rem}
    We once again fix $n$ to be $p-1$; we will continue to write expressions such as $2p(p-1)^2$ as $2pn^2$ for brevity. 
\end{rem}

We will sketch this isomorphism when $K = F$ is the maximal finite subgroup of $\S_n$, since the (collapsing) Lyndon-Hochschild-Serre spectral sequence gives an isomorphism
\[
H^*(F;(E_n)_*)^{\Gal} \simeq H^*(G;(E_n)_*).
\]

Recall that
\[
 F = \langle \zeta,\tau | \zeta^p=1,\tau^{n^2}=1,\tau^{-1} \zeta \tau = \zeta^e \rangle,
 \]
where $e \in (\Z/p)^\times$ is a generator, and we choose $\tau$ so that it is the image in $\S_n$ of $\eta = \omega^{(p^n-1)/n^2} \in \W(\F_{p^n})$, where $\omega$ is a primitive $(p^n-1)$-st root of unity. 

The calculation starts, as in~\cite{GHMR,skypaper}, by defining an action of $F$ on $\W=\W(\F_{p^n})$ by
\[
\zeta(b) = b \qquad \text{ and } \qquad \tau(b) = \omega^{\frac{p^n-1}{n^2}}b,
\]
for $b \in \W$. 

Let $\chi$ be the associated representation and denote by $\chi'$ its restriction to $C_{n^2}$, the subgroup generated by $\tau$. 

We define an $F$-module $\rho$ by the short exact sequence
\begin{equation}\label{eq:rhosequence}
0 \to \chi \to \W[F] \otimes_{\W[C_{n^2}]}\chi' \to \rho \to 0,
\end{equation}
 where the first map is given by $m \mapsto (1+\zeta+\zeta^2 + \cdots + \zeta^{p-1}) \otimes m$, for $m$ a generator of $\chi$. 
 \begin{rem}
     Justin Noel has pointed out to the author that there is a coordinate free description of the first map, namely as the canonical map
     \[
\chi \to \Coind_{C_{n^2}}^{F} \Res^{F}_{C_{n^2}} \chi.
     \]
 \end{rem}
The claim is that there is a morphism of $F$-modules $\rho \to E_{-2}$ such that there is a $F$-equivariant isomorphism
\begin{equation}\label{eq:en}
(E_n)_* \simeq S(\rho)[N^{-1}]^\wedge_I,
\end{equation}
where $N:=\prod_{g \in F} g_*(a) \in S(\rho)$, $a \in \rho$ is a generator and $I$ is the preimage of $\mathfrak{m} = (p,u_1,\ldots,u_{n-1})$ under the morphism $S(\rho)[N^{-1}] \to (E_n)_*$ (see~\cite[Lemma 3.2]{GHMR}). 

Again we simply sketch this result. Hopkins and Devinatz have shown that the divided power envelope of $E_0$ has canonical coordinates $w,w_i$ on which the action of $\S_n$ is easier to express. In fact, by~\cite[Proposition 3.3]{hd95} if we express $g \in \S_n$ by 
$g = \sum_{j=0}^{n-1} a_j S^j$, then 
\[
g(w) = a_0 w + \sum_{j=1}^{n-1} a^{\phi^j}_{n-j} ww_j,
\]
and
\begin{equation*}
\begin{split}
g(ww_i) = &pa_iw + pa^{\phi^{n-1}}_{i+1}ww_{n-1} + \cdots + pa^{\phi^{i+1}}_{n-1}ww_{i+1}  \\
	&+ a^{\phi^i}_0ww_i+ \cdots +a^\phi_{i-1}ww_1, 
\end{split}
\end{equation*}
where $a^\phi$ denotes the Frobenius. The action on the $w$ and $w_i$ is related to that of $u$ and $u_i$ by~\cite[Proposition 4.9]{hd95}, which gives (see also~\cite{naveannals})
\begin{equation}\label{eq:action1}
		g(u) \equiv a_0 u + a_{n-1}^\phi uu_1 + \cdots + a_1^{\phi^{n-1}} uu_{n-1} \mod (p,\mathfrak{m}^2)
\end{equation}
and
\begin{equation}\label{eq:action2}
g(uu_i) \equiv a_0^{\phi^i}uu_i + \cdots + a_{i-1}^\phi uu_1 \mod (p,\mathfrak{m}^2).
\end{equation}
This gives the following action for $\tau$, with $1 \le i \le n$:
\begin{equation}
		\tau(uu_i) \equiv \eta^{p^i} uu_i \mod (p,\mathfrak{m}^2),
\end{equation}
where $u_n=1$ and $u_0=p$. 

For an element $g$, of order $p$, Nave shows that we can write $g = \sum_{j=0}^\infty a_i S^i$ with $a_0 \equiv 1 \mod (p)$ and $a_1 \in \W(\F_{p^n})^\times$. With~\Cref{eq:action1,eq:action2} the action modulo $(p,\mathfrak{m}^2)$ of $\zeta$ on the ordered basis $u,uu_{n-1},\ldots,uu_{1}$,  is given by the matrix
\[
\begin{pmatrix}
1 & a_1^{\phi^{n-1}} & \cdots & a_{n-2}^{\phi^2} & a_{n-1}^\phi \\
0 & 1&  \cdots & a_{n-3}^{\phi^2} & a_{n-2}^\phi \\
\vdots & \vdots && \vdots & \vdots\\
0& 0 & \cdots & 1 & a_1^\phi \\
0 & 0 & \cdots & 0 & 1
\end{pmatrix}.
\]
Note that if we write this matrix as $A = I+B$, where $B$ is now a strictly upper triangular matrix, we see (note $B^n=0$) that
\[
u+\zeta_*(u) + \cdots + \zeta^{p-1}_*(u) \equiv 0 \mod (p,\mathfrak{m}^2).
\]

Together, these calculations, along with the sequence~\eqref{eq:rhosequence}, imply that $\rho \otimes_{\W} \F_{p^n} \simeq E_{-2} \otimes_{E_0} E_0/(p,\mathfrak{m}^2)$ as $F$-modules and that we can choose the residue class of $u$ as a generator. The idea now is to find a class $z \in E_{-2}$ with the same reduction as $u$ such that $\tau(z) = \eta z$ and $(1+ \zeta+\cdots+\zeta^{p-1})z=0$. We rely on the following result due to Nave, who credits it to Hopkins.
\begin{lemma}~\cite{nave,naveannals}
There are $z,z_1,\ldots,z_{n-1} \in (E_n)_*$ such that
\begin{enumerate}
	\item $z \equiv cu \mod (p,\mathfrak{m}^2)$ for $c \in \W^\times$,
	\item $z_i \equiv c_iuu_i \mod (p,u_1,\ldots,u_{i-1},\mathfrak{m}^2)$ for $c_i \in \W^\times$,
	\item $(1 +\zeta + \cdots + \zeta^{p-1})z = 0$,
	\item $\tau(z) = \eta z$,
	\item $(\zeta-1)z = z_{n-1} \text{ and } (\zeta-1)z_{i+1}=z_i$ for $1 \le i < n-1$,
	\item $\zeta(z_i) \equiv z_i \mod (p,z_1,\ldots,z_{i-1})$ for each $i$. 
\end{enumerate}
\end{lemma}

The morphism $\rho \to E_{-2}$ is then defined by sending the generator of $\rho$ to $z$. This defines a morphism of $\W$-algebras $S(\rho) \to E_*$. Now $S(\rho)$ is polynomial over $\W$ on $n$ generators, and we can choose a generator, $a$, such that its image in $E_*$ is the (invertible) element $z$ from above. After inverting $N$ this is an inclusion onto a dense subring; completing then gives the isomorphism described above, as in~\cite[Proposition 2]{skypaper}.

\section{The Tate cohomology of $G$}

Let $\text{tr}$ be the transfer map, defined as follows:
\[
\xymatrix@R=1mm{
	\text{tr:}M \ar[r] & M^G \\
	x \ar@{|->}[r] & \sum_{g \in G} gx.
}
\]
In what follows let $A = S(\rho)[N^{-1}]$. The following is the analogue of~\cite[Lemma 3.3]{GHMR}.
\begin{lemma}\label{lem:cpcalc}
Let $C_p \subset \S_n$ be the subgroup generated by $\zeta$. Then there is an exact sequence 
\[
A \xrightarrow{\text{tr}} H^*(C_p,A) \to \F_{p^n}[a,b,d^{\pm 1}]/(a^2) \to 0\]
with $|a| = (1,-2),|b| = (2,0)$ and $|d| = (0,-2p)$. 
\end{lemma}

\begin{proof}

Let $J$ be the $F$-module $\W[F]\otimes_{\W[C_{n^2}]}\chi'$. Choose $\W(\F_{p^n})$-generators of $J$ named $x_1,\ldots,x_p$ so that $\zeta(x_i) = x_{i+1}$ for $1 \le i \le p-1$ and $\zeta(x_p)=x_1$. Furthermore, we choose $x_1$ so that it maps to the generator $a$ of $\rho$. Specifically
\[
S(J) = \W(\F_{p^n})[x_1,x_2,\ldots,x_p].
\]

The exact sequence of~\ref{eq:rhosequence}, and the definition of the $x_i$ given above, imply that the kernel of the map $J \to \rho$ is the principal ideal generated by $\sigma_1 = x_1 +x_2 + \cdots + x_p$, and that there is a short exact sequence of $F$-modules
\begin{equation}\label{eq:SES}
0 \to S(J) \otimes \chi \to S(J) \to S(\rho) \to 0,
\end{equation}
where we set the degree of $\chi$ to be -2, so that this is an exact sequence of graded modules. Explicitly, if we identify $S(\rho)$ with $\W(\F_{p^n})[a,s_*(a),\ldots,s^{n-1}_*(a)]$, then the map $S(J) \to S(\rho)$ is defined by 
\begin{equation}\label{eq:symmap}
\begin{split}
	x_1 & \mapsto a\\
	x_i & \mapsto s_*^{i-1}(a) \text{ for } 1 \le i \le n\\
	x_p & \mapsto -(a+s_*(a)+\cdots+s_*^{n-1}(a)).
\end{split}
\end{equation}
The orbit of a monomial in $S(J)$ under the $C_p$-action has $p$ elements, unless it is a power of $\sigma_n \coloneqq x_1x_2\ldots x_p$. Therefore, as a $C_p$-module, $S(J)$ splits into a direct sum of free modules and trivial modules generated by powers of $\sigma_p$. Let $\hat{b} \in H^2(C_p,\Z_p) \subset H^2(C_p,\W(\F_{p^n})) \simeq \F_{p^n}$ be a generator, where $\W(\F_{p^n}) \subset S(J)$ is the submodule generated by the unit of the algebra. Then the description of $S(J)$ and the standard resolution of $C_p$ gives a short exact sequence
	\[
	S(J) \xrightarrow{\text{tr}} H^*(C_p,S(J)) \to \F_{p^n}[\hat b,\sigma_p] \to 0
	\]
with $| \hat b|=  (2,0)$ and $|\sigma_n| = (0,-2p)$. Similar considerations show that the cohomology of $H^*(C_p,S(J) \otimes \chi)$ is the same, however the internal degree is shifted by -2. 

We then use the long exact sequence associated to~\eqref{eq:SES} and the fact that $H^\ast(C_p,S(J))$ and $H^\ast(C_p,S(J) \otimes \chi)$ are concentrated in even degrees to get short exact sequences
\[
0 \to H^{2k-1}(C_p,S(\rho)) \to H^{2k}(C_p,S(J) \otimes \chi) \to H^{2k}(C_p,S(J)) \to H^{2k}(C_p,S(\rho)) \to 0.
\]
The middle map is zero when $k>0$, as it is induced by multiplication by $\sigma_1$, which is in the image of the transfer.  Let $d \in H^0(C_p,S(\rho)_{-2p})$ be the image of $\sigma_p$, let $b \in H^2(C_p,S(\rho)_0)$ be the image of $\hat b$ and let $a \in H^1(C_p,S(\rho)_{-2})$ map to $\tilde b \in H^2(C_p,S_0(J) \otimes \chi)$ under the boundary map. Then the above short exact sequence gives the following exact sequence in group cohomology
	\[
	S(\rho) \xrightarrow{\text{tr}} H^*(C_p,S(\rho)) \to \F_{p^n}[a,b,d]/(a^2) \to 0,
	\]
where graded commutativity forces the element $a$ to satisfy $a^2 = 0$.

We wish to identify the image of $N \in S(\rho)$ under the map $S(J) \to S(\rho)$. Under the identifications made above we see that 
\[
N = (as_*(a)s^2_*(a)\cdots s^{n-1}_*(a))(-a-s_*(a)-\cdots-s^{n-1}_*(a))^{n^2},
\]
and, by the formulas of~\ref{eq:symmap} this is the image of $d^{n^2}$. We conclude that the effect of inverting $N$ in the cohomology of $C_p$ is inversion of the element $d$. 
\end{proof}

\begin{lemma}
Let $C_p \subset \S_n$ be the subgroup generated by $\zeta$. Then there is an isomorphism 
\[
\hat H^*(C_p,A) \simeq \F_{p^n}[a,b^{\pm 1},d^{\pm 1}]/(a^2).
\] 
\end{lemma}
\begin{proof}
The map from cohomology to Tate cohomology is an isomorphism in positive dimensions and preserves the cup product structure~\cite[Exercise VI.6.1]{brown1982cohomology}. Note that the image of the transfer map is in dimension 0, and hence must be killed by multiplication by $b$ since there is no corresponding group in positive dimensions. It is well known that the Tate cohomology of a finite cyclic group is 2-periodic, with the periodicity given by the cup product.  These observations imply that we can obtain the Tate cohomology by simply inverting the degree 2 class $b$ in the ordinary group cohomology.  By definition the elements that are in the image of the transfer map are zero in Tate cohomology, and so we have:
	\[
	\hat H^*(C_p,A) \simeq \tilde H^*(C_p,A)[b^{-1}] \simeq \F_{p^n}[a,b^{\pm 1},d^{\pm 1}]/(a^2).\qedhere
	\]
\end{proof}

In order to go from $C_p$ to $F$ we need to work out the quotient actions of $\tau$ on $a,b$ and $d$. 
\begin{lemma}
The action of $\tau$ on $H^*(C_p;S(\rho)[N^{-1}])$ and $\hat H^*(C_p;S(\rho)[N^{-1}])$  is given by
\[
\tau(a) =e \eta a \quad \tau(b) =eb \quad \tau(d) =  \eta^p d,
\]
where $e \in (\Z/p)^\times$ is the element such that $\tau^{-1} \zeta \tau = \zeta^e$.
\end{lemma}
\begin{proof}
The action on $d$ is easily found as $\tau$ acts by multiplication by $\omega^{\frac{p^n-1}{n^2}}$ on each $x_i$ of $S(J)$.

The action on $b$ is more complicated. This can be found as Proposition 6.2 of Nave's thesis~\cite{nave} (see also~\cite[Example 6.7.10]{weibel}). Let $M$ be an $F$-module. First note that by naturality it suffices to find the action for the cohomology of $H^*(C_p;S(J))$. As explained in~\cite[pp. 80]{brown1982cohomology}, the action is induced by conjugation. 
Let $\alpha$ be the conjugation map $\alpha(\zeta) = \tau^{-1} \zeta \tau$ and write $M(\alpha)$ for $M$ with action induced by $\alpha$. Define a map $f:M(\alpha) \to M$ by sending $x$ to $\tau \cdot x$. Then the action is given as the composite 
\begin{equation}\label{eq:conjugation}
(\alpha,f)^*: H^*(C_p,M) \xrightarrow{\alpha^*} H^*(C_p,M(\alpha)) \xrightarrow{f_*} H^*(C_p,M)
\end{equation}
By taking two copies of the standard resolution, we can get an induced chain map between them
\[
\xymatrix{
\cdots \ar[r] & \Z[C_p] \ar@{-->}[d]^{\phi_2}\ar[r]^{N(p)} & \Z[C_p] \ar@{-->}[d]^{\phi_1}\ar[r]^{\zeta-1} & \Z[C_p] \ar[r]^{\epsilon}\ar@{-->}[d]^{\phi_0} & \Z \ar@{=}[d]\ar[r] & 0  \\
\cdots \ar[r] & \Z[C_p] \ar[r]_{N(p)} & \Z[C_p] \ar[r]_{\zeta-1} & \Z[C_p] \ar[r]_{\epsilon} & \Z \ar[r] & 0 
}
\]
with the chain map satisfying $\phi_*(gx) = \alpha(g)\phi_*(x)$. Here
$N(k) = 1+ \zeta + \cdots + \zeta^{k-1}$. Nave calculates that the first few maps are 	$\phi_0 = \alpha$  and $\phi_1(g) = \phi_2(g) = N(e)\alpha(g)$.

To compute $(\alpha,f)^\ast$ (see~\Cref{eq:conjugation}) we first apply $\Hom_{\Z[C_p]}(-;M)$ to the top row and $\Hom_{\Z[C_p]}(-;M(\alpha))$ to the bottom row, and then follow with $\Hom_{\Z[C_p]}(-;f)$  to get the following (cf.~\cite[pp. 22]{nave}):
\[
\xymatrix{
M \ar[r]^{\zeta-1}  \ar[d]^{\tau} & M \ar[r]^{N(p)}\ar[d]^{\tau N(e)} & M \ar[d]^{\tau N(e)} \ar[r] &\cdots \\
M \ar[r]_{\zeta-1} & M \ar[r]_{N(p)} & M \ar[r]& \cdots \\
}
\]
Taking $M = S(J) = \W[x_1,\ldots,x_p]$ and tracing what happens to $b$ we find that $\tau(b) = eb$.   

The action of $a$ then follows when we recall it is the preimage of $b$ with respect to the isomorphism
\[
H^1(C_p,S(\rho)) \to H^2(C_p;S_0(J) \otimes \chi)
\]
and so the action is twisted by the representation $\chi$. 
\end{proof}
 Before we can calculate the invariants we need to obtain a relationship between $\eta^{p-1}$ (recall that $\eta$ is the image of $\tau$) and $e$. 
\begin{lemma}\cite[Page 24]{nave}
	\[
	\eta^{p-1} \equiv e \mod (p)
	\]
\end{lemma}
\begin{proof}
	Since $\zeta$ has order $p$, we have $\zeta = 1+a_1S \mod (S^2)$ where $a_1$ is a unit in $\W^\times$ (recall our description of elements of order $p$ from~\Cref{sec:eng}). Then we have (working modulo $S^2$)
	\[
	\zeta^e  = (1+a_1eS) = \tau^{-1} \zeta \tau = 1 + \tau^{-1} a_1S \tau = 1+\tau^{p-1}a_1S.
	\]
	It follows that $a_1 \tau^{p-1} \equiv a_1 e \mod (p)$ and the result follows. 
\end{proof}

\begin{prop}\label{prop:cohomcalcs}
	There is an isomorphism
	\[
	\hat H^*(G;(E_n)_*) \simeq \F_p[\alpha,\beta^{\pm 1},\Delta^{\pm 1}]/(\alpha^2)
	\]
	and an exact sequence
	\[
	E_* \xr{tr} H^*(G;(E_n)_*) \to \F_p[\alpha,\beta,\Delta^{\pm 1}]/(\alpha^2) \to 0.
	\]
	
\end{prop}
\begin{proof}
We first need to pass from $\hat H^*(C_p;A)$ to $\hat H^*(F;(E_n)_*)$. We do this by using the (collapsing) Lyndon--Hochschild--Serre spectral sequence.  First note that $\hat H^*(F;A) \simeq \hat H^*(C_p,A)^{\langle \tau \rangle}$ since the composite of the restriction and transfer maps to the trivial group is multiplication by $n^2$, which is a unit in $\W(\F_{p^n})$. 

Now $\tau(d^k) = \eta^{pk}d^k$, which is equal to $d^k$ if and only if $n^2$ divides $k$.
Define elements
\begin{align*}
	\alpha &\coloneqq  \frac {a}{d} \in \hat H^1(C_3,(S(\rho)[N^{-1}])_{2p-2}) \\
	\beta &\coloneqq  \frac{b}{d^n} \in \hat H^2(C_3,(S(\rho)[N^{-1}])_{2pn})\\
	\Delta &\coloneqq \frac{1}{d^{n^2}} \in \hat H^0(C_3,(S(\rho)[N^{-1}])_{2pn^2}).
\end{align*}
It is easy to check that these are invariant under the action of $\tau$ and an argument using, for example,~\cite[Lemma 4.1]{MR2066512}, shows that $\beta,\Delta$, and their inverses, along with $\alpha$, generate the invariants. Thus we see that $\hat{H}^*(F,A) \simeq \F_{p^n}[\alpha,\beta^{\pm 1},\Delta^{\pm 1}]/(\alpha^2)$. A completion argument, as given in~\cite[Theorem 6]{skypaper} implies that $\hat H^*(F,(E_n)_*) \simeq \hat H^*(F,S(\rho)[N^{-1}]^\wedge_I)) \simeq \hat H^*(F,S(\rho)[N^{-1}]$.

Finally we need to pass from $\hat H^*(F,(E_n)_*)$ to $H^*(G,(E_n)_*)$. We claim that the effect of the Galois action is only to reduce coefficients to $\F_p$. Indeed, as explained in~\cite[pp. 490]{MR2066512}, in bidegree $(s,t)$ the cohomology group is either $\F_{p^n}$ or 0, and  the invariants are either $\F_p$ or 0. We can always replace the generators $\Delta,\alpha$ and $\beta$ by any non-zero element of the cohomology group they appear in, and this implies that they are invariant under the Galois action; it follows that
\[ 
\hat H^*(G,(E_n)_*) \simeq \F_p[\alpha,\beta^{\pm 1},\Delta^{\pm 1}]/(\alpha^2)
\]

The group cohomology calculation follows immediately. 
\end{proof}
We are now in position to prove the main result of this note.
\begin{theorem}\label{thm:tatecontractible}
 	The Tate spectrum $E_n^{tG}$ is contractible. 
 \end{theorem} 
\begin{proof}

We have calculated the $E_2$-term of the spectral sequence above. The differentials in the Tate spectral sequence can be inferred from the differentials in the homotopy fixed point spectral sequence, which are now well known. In particular $\alpha$ and $\beta$ denote the image, under the unit map $S^0 \to E_n^{hG}$, of $\alpha_1 \in \pi_{2p-3}S^0$ and $\beta_1 \in \pi_{2(p^2-p-1)}S^0$. 

Note that in the spectral sequence $E_2^{s,t}$ is concentrated in degrees $t \equiv 0 \mod (2(p-1))$; this implies that $d_r =0$ unless $r \equiv 1 \mod (2(p-1))$. For example, the first possible differential is a $d_{2n+1}$. Toda~\cite{toda} has shown that $\alpha_1 \beta_1^p =0 \in \pi_*(S^0)$; thus we must have $\alpha \beta^p=0$ in our spectral sequence, and the only possibility is that $d_{2n+1}(\Delta \beta) = \alpha \beta^{p}$, or in other words, $d_{2n+1}(\Delta) = \alpha \beta^{p-1}$ (at least up to multiplication by a unit in $\F_{p^n}$).

Likewise Toda's relation $\beta_1^{pn+1} = 0 \in \pi_*(S^0)$ implies that $\beta^{pn+1}$ must die; sparsity again implies that the only possible differential is $d_{2n^2+1}(\Delta^n \alpha \beta^n) = \beta^{pn+1}$.  This gives the differential $d_{2n^2+1}(\Delta^n \alpha) = \beta^{n^2+1}$. The homotopy fixed point spectral sequence collapses at this stage for degree reasons. In the Tate spectral sequence we have that $\beta$ is invertible, and so the second differential implies that everything in the spectral sequence dies, and indeed $E_n^{tG} \simeq \ast$. This is shown in~\Cref{fig:tSS} for $n=2$. 

\begin{figure}[tbh]
\centering
\includegraphics[scale=0.97]{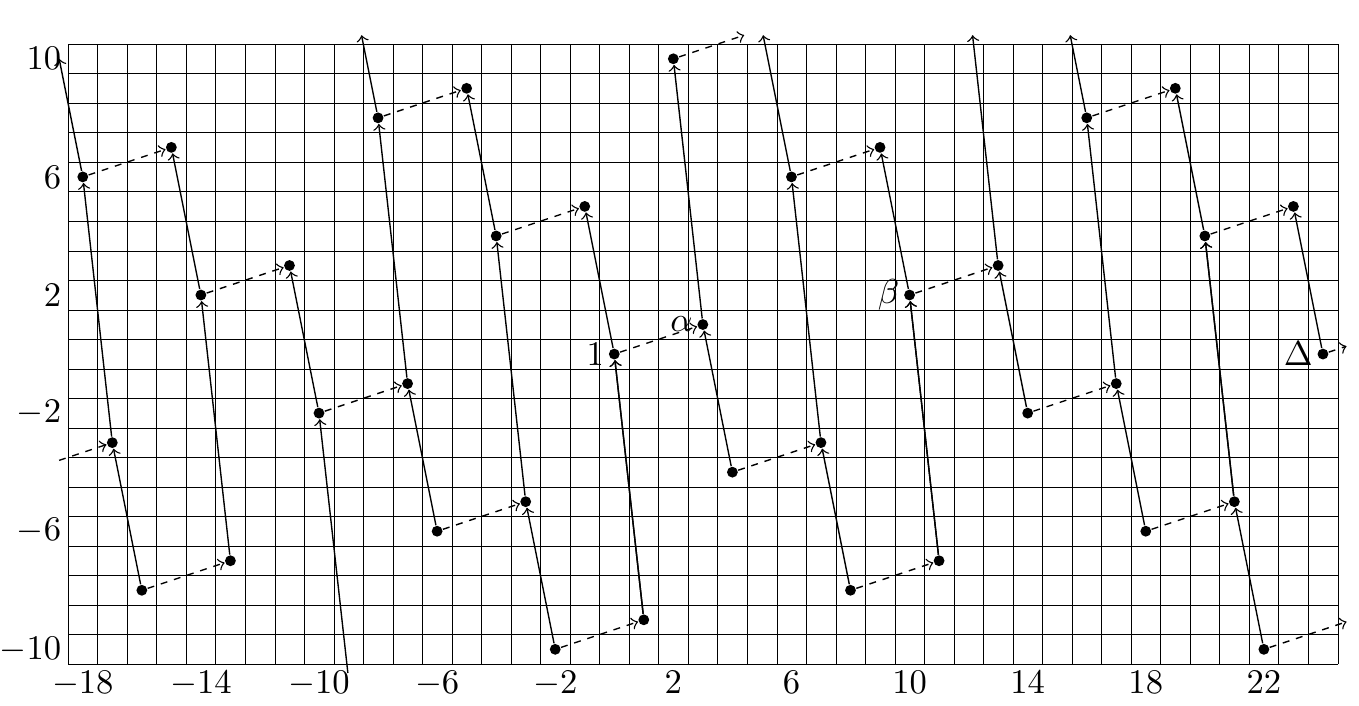}
\caption{The Tate spectrum sequence to compute $E_n^{tG}$, shown for $n=2,p=3$. Here the dashed lines representation multiplication by $\alpha$ and the vertical lines are differentials.}\label{fig:tSS}
\end{figure}
\end{proof}
\begin{rem}
   The key point is that the nilpotent element $\beta \in \pi_*S^0$ is detected by an invertible class in Tate cohomology . Similarly in~\cite{koduality} the contractibility of $KU^{tC_2}$ was because the nilpotent element $\eta$ was inverted in the $E_2$-term of the Tate spectral sequence. This suggests a possible approach to generalising these arguments. 
\end{rem}
We have also computed the following.
\begin{lemma}\label{lem:htpyftptcalcs}
In the spectral sequence $H^*(G,(E_n)_*) \Rightarrow \pi_*(E_n^{hG})$ we have (in positive cohomological degrees)
\[
E_\infty^{\text{even},\ast} \simeq \frac{\F_{p}[\Delta^{\pm p}][\beta]}{\beta^{n^2+1}} \quad \text{ and }\quad E_\infty^{\text{odd},\ast} \simeq \frac{\alpha \F_p[\Delta^{\pm p}]\{1,\Delta,\cdots,\Delta^{p-2}\}[\beta]}{\beta^n}.
\]
\end{lemma}
\begin{rem}
\begin{enumerate}
    \item We can say slightly more about the image of the transfer, again referring to Nave for the proof. There is a natural map
    \[
\F_p[\Delta^{\pm 1}] \to H^0(G,(E_n)_*)/p
    \]
    which is a split injection. If we let $M$ denote the other summand determined by the splitting, then
    \[
H^*(G,(E_n)_*) = \frac{H^0(G,(E_n)_*)[\alpha,\beta]}{(p\alpha)(p\beta)(\alpha^2)(M\alpha)(M\beta)}.
    \]
    \item In fact the previous two results holds true for more than just the group $G$. Let $Z \subset G$ be the group generated by $\tau^n$. Then the theorem is in fact true for any group which contains the subgroup $Z$ and an element of order $p$. See~\cite[Lemma 23]{HennFinite} and the discussion therein.
    \item   Note that $E_n^{hG}$ is periodic of order $2p^2n^2$, and that $\Delta^p$ detects a periodicity class. 
\end{enumerate}

 
\end{rem}

We give two simple corollaries of~\Cref{thm:tatecontractible}. The first follows essentially by definition. 
\begin{cor}
	The norm map $N(E_n):(E_n)_{hG} \to E_n^{hG}$ is an equivalence.
\end{cor}

In the following the duality functor we are interested in is $K(n)$-local Spanier-Whitehead duality; that is $DX = F(X,L_{K(n)}S^0)$. We learnt the following technique from~\cite[Proposition 2.5.1]{behrensmod}. 	
\begin{cor}
	$E_n^{hG}$ is $K(n)$-locally self-dual up to suspension. In fact $DE_n^{hG} \simeq \Sigma^N E_n^{hG}$ where $N \equiv -n^2 \mod (2pn^2)$ and $N$ is only uniquely defined modulo $2p^2n^2$.  
\end{cor}
\begin{proof}
	By definition and the contractibility of the Tate spectrum we have
\[
	D(E_n^{hG}) \simeq F(E_n^{hG},L_{K(n)}S^0) \simeq F((E_n)_{hG},L_{K(n)}S^0) \simeq F(E_n,L_{K(n)}S^0)^{hG} \simeq (DE_n)^{hG},	
\]
and so there is an associated homotopy fixed point spectral sequence
\begin{align}\label{eq:ss}
H^s(G,(DE_n)_t)) \Rightarrow \pi_{t-s}(DE_n)^{hG}.
\end{align} 
By~\cite{strcklandgh} there is an isomorphism $DE_n \simeq \Sigma^{-n^2} E_n$ and so the $E_2$-term of the spectral sequence~\eqref{eq:ss} can be written as 
\[
H^*(G,(DE_n)_*) \simeq H^*(G,\Sigma^{-n^2} E(_n)_*).
\]
This is a spectral sequence of modules over the spectral sequence of algebras
\begin{align}\label{eq:ss2}
H^s(G,(E_n)_t) \Rightarrow \pi_{t-s}(E_n^{hG}),
\end{align}
and so the spectral sequence of~\eqref{eq:ss} is isomorphic to the spectral sequence~\eqref{eq:ss2} up to a shift congruent to $-n^2$ modulo $2pn^2$ (the periodicity of the $E_2$-term). In other words, there exists a map 
$$S^N \to D(E_n^{hG}),$$ which is detected on the zero line of the spectral sequence~\eqref{eq:ss} which extends to an equivalence
\[
\Sigma^N E_n^{hG} \xrightarrow{\simeq} D(E_n^{hG}),
\]
where $N \equiv -n^2 \mod (2pn^2) $, and $N$ is only uniquely determined modulo $2p^2n^2$, the periodicity of $E_n^{hG}$. 
\end{proof}
\begin{rem}
  Since there is an equivalence $F(L_{K(n)}X,L_{K(n)}S^0) \simeq F(X,L_{K(n)}S^0)$ for any $X$, the full strength of~\Cref{thm:tatecontractible} is not needed, and this result could also be deduced from the known fact that there is a $K(n)$-local equivalence between $E_n^{hG}$ and $(E_n)_{hG}$. In fact, a similar argument proves that $DE_n^{hK}$ is self-dual, for any finite $K \subset \G_n$, up to a shift. In fact $H^*(K,(E_n)_*)$ is always periodic, of some period a multiple of 2, and $DE_n^{hK}$ is isomorphic to $\Sigma^N E_n^{hK}$, where $N$ is congruent to $-n^2$ modulo this periodicity. If the order of $K$ does not divide $p$, then $DE_n^{hK} \simeq \Sigma^{-n^2}E_n^{hK}$. 
\end{rem}
\begin{example}
	By Hahn and Mitchell~\cite{hahnmitchell}, or the author and Stojanoska~\cite{koduality}, there is an isomorphism $DE_1^{C_2} \simeq \Sigma^{-1} E_1^{C_2}$, whilst Behrens~\cite{behrensmod} has calculated that $DE_2^{hG_{24}} \simeq \Sigma^{44} E_2^{hG_{24}}$.
\end{example}
The fact that $DE_2^{hG_{24}} \simeq \Sigma^{44} E_2^{hG_{24}}$ and not $\Sigma^{-2}E_2^{hG_{24}}$ shows that the equivalence $DE_n \simeq \Sigma^{-n^2}E_n$ is not $\G_n$-equivariant on the point-set level, as is noted in~\cite[Section 7]{davisorbits}.

\bibliographystyle{amsalpha}
    \bibliography{../BibtexDatabase}
\end{document}